\documentclass[11pt]{article}

\usepackage{amsmath,amsfonts,amsthm,amssymb}
\usepackage[margin=1in,letterpaper]{geometry}
\usepackage{kbordermatrix}
\usepackage{physics}
\usepackage{graphicx}
\usepackage{tikz}
\usepackage{amsthm}
\theoremstyle{definition}
\newtheorem{defi}{Definition}[section]
\newtheorem{theorem}[defi]{Theorem}
\newtheorem{corollary}[defi]{Corollary}
\newtheorem{lemma}[defi]{Lemma}
\newtheorem{prop}[defi]{Proposition}

\newtheorem{eg}[defi]{Example}
\usepackage{bm}

\DeclareMathOperator{\spa}{span}

\newcommand{\aff}{\mathrm{aff}}

\DeclareMathOperator{\rd}{rd}
\DeclareMathOperator{\gd}{gd}

\DeclareMathOperator{\diag}{diag}

\DeclareMathOperator{\si}{si}
\DeclareMathOperator{\Aut}{Aut}

\DeclareMathOperator{\sign}{sign}

\title{Realizable Dimension of Periodic Frameworks}
\author{Ryoshun Oba\thanks{Department of Mathematical Informatics, Graduate School of Information Science and Technology, University of Tokyo, Tokyo 113-8656, Japan. Email: \texttt{ryoshun\_oba@mist.i.u-tokyo.ac.jp}}
\and Shin-ichi Tanigawa\thanks{Department of Mathematical Informatics, Graduate School of Information Science and Technology, University of Tokyo, Tokyo 113-8656, Japan. Email: \texttt{tanigawa@mist.i.u-tokyo.ac.jp}}}

\begin{document}
\maketitle
\begin{abstract}
    Belk and Connelly introduced the realizable dimension $\rd(G)$ of a finite graph $G$, which is the minimum nonnegative integer $d$ such that every framework $(G,p)$ in any dimension admits a framework in $\mathbb{R}^d$ with the same edge lengths.
    They characterized finite graphs with realizable dimension at most $1$, $2$, or $3$ in terms of forbidden minors.
    In this paper, we consider periodic frameworks and extend the notion to $\mathbb{Z}$-symmetric graphs. We give a forbidden minor characterization of $\mathbb{Z}$-symmetric graphs with realizable dimension at most $1$ or $2$, and show that the characterization can be checked in polynomial time for given quotient $\mathbb{Z}$-labelled graphs.
\end{abstract}
\section{Introduction}

A \emph{framework} in the Euclidean space $\mathbb{R}^d$ is a graph in which each vertex is a point in $\mathbb{R}^d$ and each edge is a straight line connecting the endvertices.
It is denoted by a pair $(G,p)$, where $G$ is the underlying (possibly infinite) undirected graph and $p:V(G)\rightarrow \mathbb{R}^d$ is a map called a point configuration.
The \emph{affinie dimension} of a framework $(G,p)$ is defined to be the affine dimension of $p(V(G))$.
Two frameworks $(G,p)$ and $(G,q)$, possibly in different dimensions, are \emph{equivalent} if $\|p(i)-p(j)\|=\|q(i)-q(j)\|$ for all $ij \in E(G)$, where $\|\cdot\|=\|\cdot\|_2$ is the Euclidean norm.
Belk and Connelly~\cite{belk2007realizability} introduced a minor monotone parameter called the realizable dimension of finite graphs.
A finite graph $G$ is \emph{$d$-realizable} if every framework $(G,p)$ in any dimension admits an equivalent framework with the affine dimenision at most $d$.
The \emph{realizable dimension} $\rd(G)$ of a finite graph $G$ is the smallest nonegative integer $d$ such that $G$ is $d$-realizable. 
Belk and Connelly characterized finite graphs satisfying $\rd(G) \leq d$ in terms of forbidden minors for $d=1,2,3$ \cite{belk2007threedim,belk2007realizability}.
The lists of forbidden minors are $\{K_3\}$ when $d=1$, $\{K_4\}$ when $d = 2$, and $\{K_5,K_{2,2,2}\}$ when $d=3$.

In this paper, we extend the notion of realizable dimension to a periodic setting.
The rigidity of periodic frameworks has been widely investigated in the past decade \cite{borcea2010periodic,kaszanitzky2021global,malestein2013generic,MT14}. See~\cite{schulze2017rigidity} for further reference.
Since our aim in this paper is to examine the extendability of the concept to periodic frameworks, here we focus on the simplest setting, where frameworks have one dimensional periodicity (see Figure~\ref{fig:1}) and motions of frameworks are subject to a given fixed periodicity.
An infinite graph $G$ is \emph{$\mathbb{Z}$-symmetric} if an additive group $\mathbb{Z}$ acts freely on $\Aut(G)$ and the number of vertex orbits is finite. 
We fix an action and the action of $\gamma \in \mathbb{Z}$ on $V(G)$ is denoted by $+\gamma$.
Two $\mathbb{Z}$-symmetric graphs are isomorphic if there is an ordinary graph isomorphism compatible with the $\mathbb{Z}$-actions.
For a nonzero vector $\ell \in \mathbb{R}^d$, a framework $(G,p)$ in $\mathbb{R}^d$ is called \emph{$\ell$-periodic} if $G$ is a $\mathbb{Z}$-symmetric graph and $p(v+1)-p(v)=\ell$ for all $v \in V(G)$. The vector $\ell$ is called the \emph{lattice vector} of $(G,p)$.
A $\mathbb{Z}$-symmetric graph $G$ is \emph{$d$-realizable} if every $\ell$-periodic framework $(G,p)$ in any dimension with any lattice vector $\ell$ admits an equivalent $\ell$-periodic framework with the affine dimension at most $d$.
The \emph{realizable dimension} $\rd(G)$ of a $\mathbb{Z}$-symmetric graph $G$ is defined to be the smallest nonnegative integer $d$ such that $G$ is $d$-realizable.
We show that the parameter $\rd(\cdot)$ is minor monotone, where a minor operation of a $\mathbb{Z}$-symmetric graph is defined as a deletion or contraction of a vertex orbit or an edge orbit which does not destroy $\mathbb{Z}$-symmetry. Our main theorem is a forbidden minor characterization of $\mathbb{Z}$-symmetric graphs with realizable dimension at most $1$ or $2$.
Our proof of the characterizations leads to a polynomial time algorithm to check the $1$-realizability or $2$-realizability of $\mathbb{Z}$-symmetric graphs when a $\mathbb{Z}$-symmetric graph is given in terms of its quotient $\mathbb{Z}$-labelled graph.

\begin{figure} \label{fig:1}
    \centering


\begin{tikzpicture}
    \tikzset{vertex/.style={draw,fill=black,circle,inner sep=2pt,minimum size=3pt}}
    \node[vertex] (a2) at (3,0) {};
    \node[vertex] (b2) at (3,2) {};
    \node[vertex] (c2) at (1,1) {};
    \node[vertex] (a3) at (6,0) {};
    \node[vertex] (b3) at (6,2) {};
    \node[vertex] (c3) at (4,1) {};
    \draw (a2) to (b2);
    \draw (a2) to (c2);
    \draw (b2) to (c2);
    \draw (a3) to (b3);
    \draw (b3) to (c3);
    \draw (c3) to (a3);
    \draw (c2) to (b3);
    \draw (a2) to (b3);
    \node (a1) at (0.5,0) {};
    \node (b1) at (0.5,2) {};
    \node (c1) at (0.5,1) {};
    \node (a4) at (6.5,0) {};
    \node (b4) at (6.5,2) {};
    \node (c4) at (6.5,1) {};
    \node (d1) at (0.5,1.5) {};
    \draw (d1) to (b2);
    \node (e1) at (0.5,1/3) {};
    \draw (e1) to (b2);
    \node (d4) at (6.5,1.5) {};
    \node (e4) at (6.5,1/3) {};
    \draw (d4) to (c3);
    \draw (e4) to (a3);

    \node (l1) at (3,2.5) {};
    \node (l2) at (6,2.5) {};
    \draw[->,>=latex] (l1) to node [midway, above] {$\ell$} (l2);
\end{tikzpicture}
\qquad
\begin{tikzpicture}
    \tikzset{vertex/.style={draw,fill=black,circle,inner sep=2pt,minimum size=3pt}}
    \node[vertex] (a) at (0,1) {};
    \node[vertex] (b) at (2,2) {};
    \node[vertex] (c) at (2,0) {};
    \draw[->,>=latex] (a) to [bend left=20] node [midway,above] {$+0$} (b);
    \draw[->,>=latex] (a) to [bend right=20] node [midway,above] {$+1$} (b);
    \draw[->,>=latex] (c) to [bend right=20] node [midway,right] {$+0$} (b);
    \draw[->,>=latex] (c) to [bend left=20] node [midway,left] {$+1$} (b);
    \draw[->,>=latex] (a) to node [midway,below] {$+0$} (c);
\end{tikzpicture}
\caption{An affine $2$-dimensional periodic framework with the lattice vector $\ell$ (left) and the quotient $\mathbb{Z}$-labelled graph of its underlying $\mathbb{Z}$-symmetric graph (right)}
\end{figure}
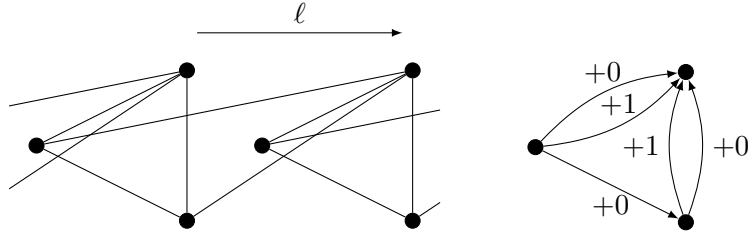

In the case of finite graphs, the characterization of $d$-realizable graphs is done by first enumerating obstacles of $d$-realizability, that is, graphs whose realizable dimension is more than $d$, and then by showing that graphs that avoid these obstacles as a minor have a nice inductive construction~\cite{belk2007realizability}.
For proving that $G$ is an obstacle of $d$-realizability, it is enough to find a $d+1$-dimensional globally rigid realization $(G,p)$ of $G$, i.e., $(G,p)$ has no equivalent but non-congruent framework in $\mathbb{R}^{d+1}$.
A useful sufficient condition of the global rigidity of frameworks is super stability by Connelly \cite{connelly1982rigidity}, which actually guarantees stronger rigidity property called universal rigidity.
A modern exposition on the connection between universal rigidity and super stability is given through a semidefinite programming (SDP) formulation \cite{SY07}. 
Motivated by this background, in this paper we extend the super stability condition to the periodic setting based on an unpublished work~\cite{sean}. Then periodic super stability is used to construct obstacles of $d$-realizability in the periodic setting.

The paper is organized as follows.
In Section 2, we give preliminary facts on $\mathbb{Z}$-labelled graphs.
In Section 3, we formally introduce periodic frameworks and the periodic extension of realizable dimension.
In Section 4, we investigate affine images of frameworks as a tool to find lower affine dimensional equivalent frameworks.
In Section 5, equilibrium stresses of periodic frameworks are introduced and a relationship between periodic universal rigidity and periodic super stability is discussed. 
In Section 6, we determine the realizable dimension of $\mathbb{Z}$-symmetric graphs such that for every pair of vertex orbits there is an edge orbit between them.
In Section 7, we introduce two operations to obtain $d$-realizable $\mathbb{Z}$-labelled graphs from smaller $d$-realizable $\mathbb{Z}$-labelled graphs. Then we characterize $\mathbb{Z}$-symmetric graphs with realizable dimension at most $1$ or $2$, and give a polynomial time algorithm.

\section{$\mathbb{Z}$-labelled Graphs} \label{sec:rd2}
Throughout the paper we assume that $\mathbb{Z}$ is an additive group.
For two edges of a directed multigraph, we say that they are parallel if they have the same set of endvertices.
For two parallel edges $e, f$ of a directed multigraph, we say that $e$ and $f$ have the same direction if the heads of $e$ and $f$ coincide, and have the inverse directions if the head of $e$ and the tail of $f$ coincide.
Let $K_n$ be a simple undirected complete graph on $n$ vertices.
By abuse of notation, a directed graph obtained by orienting each edge of $K_n$ is also denoted as $K_n$.

\subsection{$\mathbb{Z}$-symmetric graphs and $\mathbb{Z}$-labelled graphs}
Let $G=(V,E)$ be a $\mathbb{Z}$-symmetric graph.
A vertex (resp. edge) orbit is an orbit of the $\mathbb{Z}$-action on $V$ (resp. $E$), and the set of vertex (resp. edge) orbits is denoted as $V/\mathbb{Z}$ (resp. $E/\mathbb{Z}$).
We always assume that $V/\mathbb{Z}$ and $E/\mathbb{Z}$ are finite sets.
In a $\mathbb{Z}$-symmetric graph $G$, a vertex subset $\hat{V}$ is called a \emph{representative vertex set} if it is a transversal of the set of vertex orbits of $G$, i.e., $\hat{V}$ has exactly one element from each vertex orbit. The unique element of a vertex orbit is called the \emph{representative vertex} of the orbit (in $\hat{V}$).
Observe that a $\mathbb{Z}$-symmetric graph can be represented as a directed multigraph with an edge $\mathbb{Z}$-labelling as follows.
For a vertex representative set $\hat{V}$, each edge orbit is written as $\{\{i+\beta,j+\beta+\gamma\}:\beta \in \mathbb{Z}\}$ for some $i,j \in \hat{V}$ and $\gamma \in \mathbb{Z}$. Regard this edge orbit as a directed edge from $i$ to $j$ with a label $\gamma$. 
The resulting pair of a directed multigraph and an edge $\mathbb{Z}$-labelling is called the \emph{quotient $\mathbb{Z}$-labelled graph} of $G$ with respect to $\hat{V}$ (see Figure~\ref{fig:1} for an example).

In general, a \emph{$\mathbb{Z}$-labelled graph} (or \emph{$\mathbb{Z}$-gain graph}) is a pair $(\hat{G},z)$ of a finite directed multigraph $\hat{G}=(\hat{V},\hat{E})$ and a map $z:\hat{E} \rightarrow \mathbb{Z}$. We allow $\hat{G}$ to have selfloops.
$z$ is called an \emph{(edge) labelling}, and it is \emph{simple} if selfloops have nonzero labels and no pair of parallel edges has the same direction and the same label nor has the inverse directions and inverse labels.
We always assume that the edge labelling of a $\mathbb{Z}$-labelled graph is simple, and we write an edge $e$ from $i$ to $j$ with a label $z(e)$ as $(i,j;z(e))$.
A vertex (resp. edge) set of $\mathbb{Z}$-labelled graph $(\hat{G},z)$ is denoted by $\hat{V}(\hat{G})$ (resp. $\hat{E}(\hat{G})$).
For a $\mathbb{Z}$-labelled graph $(\hat{G},z)$, its \emph{lift graph} (or \emph{covering graph}) is an undirected graph $G=(V,E)$, where $V=\hat{V}(\hat{G}) \times \mathbb{Z}=\{i+\gamma:i \in \hat{V}(\hat{G}), \gamma \in \mathbb{Z}\}$ and
\begin{align*}
    E =\{ \{i+\gamma,j+(\gamma+z(e))\} : (i,j;z(e)) \in \hat{E}(\hat{G}), \gamma \in \mathbb{Z} \}.
\end{align*}
The lift graph of $\mathbb{Z}$-labelled graph is $\mathbb{Z}$-symmetric. 
We remark that the edge set of a $\mathbb{Z}$-labelled graph is identified with the set of edge orbits of a $\mathbb{Z}$-symmetric graph.

We define the notion of isomorphism of $\mathbb{Z}$-labelled graphs.
An \emph{edge inversion} of a $\mathbb{Z}$-labelled graph is an operation in which an orientation of an edge is inverted and its label is also inverted.
For a $\mathbb{Z}$-labelled graph $(\hat{G},z)$, a \emph{switching} of $i \in \hat{V}$ by $\gamma \in \mathbb{Z}$ changes each edge $(i,j;z(e)))$ to $(i,j;z(e)+\gamma)$, and each edge $(j,i;z(e))$ to $(j,i;z(e)-\gamma)$. The labels of selfloops incident to $i$ are unchanged.
$\mathbb{Z}$-labelled graphs $(\hat{G},z)$ and $(\hat{H},w)$ are \emph{isomorphic} if after a sequence of edge inversions and switchings, there is a bijection $\varphi:\hat{V}(\hat{G}) \rightarrow \hat{V}(\hat{H})$ such that $(i,j;\gamma) \in \hat{E}(\hat{G})$ if and only if $(\varphi(i),\varphi(j);\gamma) \in \hat{E}(\hat{H})$.
$\mathbb{Z}$-labelled graphs are isomorphic if and only if their lift graphs are isomorphic $\mathbb{Z}$-symmetric graphs.
To see an example, let $K_2^\bullet$ be a directed multigraph on $\{1,2\}$ with a pair of parallel edges $e,f$ from $1$ to $2$. 
Then, $(K_2^\bullet,z)$ and $(K_2^\bullet,z')$ are isomorphic if and only if $z(e)-z(f)=\pm (z'(e)-z'(f))$.

We introduce minor operations for $\mathbb{Z}$-labelled graphs.
Let $(\hat{G},z)$ be a $\mathbb{Z}$-labelled graph.
An \emph{edge (resp. vertex) deletion} is an operation in which an edge $e$ (resp. a vertex $i$) is deleted and the labelling is restricted to the remaining edges.
The resulting $\mathbb{Z}$-labelled graph is denoted as $(\hat{G}-e,z)$ (resp. $(\hat{G}-i,z)$).
An \emph{edge contraction} of an edge $e=(i,j;0)$ is an operation in which endvertices $i$ and $j$ are identified and the labelling is restricted to the remaining edges, and if any, selfloops with zero label or some parallel edges are deleted to make the edge labelling simple\footnote{Although there may be multiple ways to delete parallel edges, the resulting $\mathbb{Z}$-labelled graphs are isomorphic, so one may delete any edges.}.
The resulting $\mathbb{Z}$-labelled graph is denoted as $(\hat{G}/e,z)$.
Note that we define an edge contraction only along an edge with $0$ label. Note also that one can always convert the labelling of an edge, which is not a selfloop, to $0$ by a switching.
A deletion or contraction of a vertex and an edge in a $\mathbb{Z}$-labelled graph is interpreted in terms of the lift graph as was described in the introduction. 
Let $(\hat{G},z)$ and $(\hat{H},w)$ be $\mathbb{Z}$-labelled graphs.
$(\hat{G},z)$ is a subgraph of $(\hat{H},w)$ if a $\mathbb{Z}$-labelled graph isomorphic to $(\hat{G},z)$ is obtained from $(\hat{H},w)$ by a sequence of edge deletions and vertex deletions. 
$(\hat{G},z)$ is a \emph{minor} of $(\hat{H},w)$ if a $\mathbb{Z}$-labelled graph isomorphic to $(\hat{G},z)$ is obtained from $(\hat{H},w)$ by a sequence of edge deletions, vertex deletions and edge contractions.
$(\hat{G},z)$ is a minor of $(\hat{H},w)$ if and only if the lift graph of $(\hat{G},z)$ is a minor of the lift graph of $(\hat{H},w)$ as a $\mathbb{Z}$-symmetric graph.

A \emph{walk} in a $\mathbb{Z}$-labelled graph $(\hat{G},z)$ is an alternating sequence $v_1e_1v_2\cdots e_kv_{k+1}$ of vertices and edges of $\hat{G}$ such that $v_i$ and $v_{i+1}$ are the endvertices of $e_i$, and it is \emph{closed} if $v_{k+1}=v_1$.
A closed walk $v_1e_1v_2\cdots e_kv_1$ is a \emph{cycle} if $v_1, \ldots, v_k$ are distinct.
For a closed walk $v_1e_1v_2\cdots e_kv_1$, its \emph{gain} is defined as $\sum_{i=1}^k \sign(e_i)z(e_i)$, where $\sign(e_i)=1$ if $e_i$ is forward direction and $\sign(e_i)=-1$ otherwise. Note that the gain of a closed walk is well-defined as it is invariant by edge inversions or switchings.
A closed walk with zero (resp. nonzero) gain is called \emph{balanced} (resp. \emph{unbalanced}).
A $\mathbb{Z}$-labelled graph is \emph{balanced} if every closed walk, or equivalently every cycle, is balanced.
Balancedness is preserved by minor operations.
A balanced $\mathbb{Z}$-labelled graph is isomorphic to a $\mathbb{Z}$-labelled graph with all labels equal to $0$.
See, e.g., \cite{kaszanitzky2021global} for more details.

\section{Periodic Frameworks and the Realizable Dimension of $\mathbb{Z}$-symmetric Graphs}
In this secton we introduce basic notation for periodic frameworks.
Note that throughout the paper we distinguish an affine $d$-dimensional framework and a framework in $\mathbb{R}^d$.

\subsection{Fixed-lattice model}
Recall that in this paper we consider frameworks with one-dimensional periodicity and assume that the underlying lattice is fixed.
In order to visualize the effect of the fixed lattice, we introduce additional notations which are special in this paper.

Let $(\hat{G}^\circ,z)$ be a $\mathbb{Z}$-labelled graph obtained from $(\hat{G},z)$ by adding a new vertex $L$ not in $\hat{V}$ and a selfloop $e_L$ incident to $L$ with a label $+1$. The edge set $\hat{E}\cup\{e_L\}$ of $\hat{G}^\circ$ is denoted as $\hat{E}^\circ$.
The addition of the selfloop $e_L$ to $(\hat{G},z)$ results in the addition of a new exceptional edge orbit to the lift graph $G$ of $(\hat{G},z)$. The new edge orbit is denoted as $E_L$. Then there is an obvious correspondence between
$(E \cup E_L)/\mathbb{Z}$ and $\hat{E}^\circ \cup \{e_L\}$. Hence, in the subsequent discussion, we sometimes identify $(E \cup E_L)/\mathbb{Z}$ and $\hat{E}^\circ \cup \{e_L\}$ if it is clear from the context.

A framework is called \emph{periodic} if it is $\ell$-periodic for some vector $\ell$.
Since our focus in this paper is the model in which the underlying lattice length is fixed, we say that two periodic frameworks $(G,p)$ and $(G,q)$ are \emph{equivalent} if $\|p(u)-p(v)\|=\|q(u)-q(v)\|$ for all $uv \in E(G)$ and their lattice vectors have the same lengths. 
Then a $\mathbb{Z}$-symmetric graph $G$ is $d$-realizable if and only if every periodic framework $(G,p)$ in any dimension admits an equivalent periodic framework in $\mathbb{R}^d$.

\subsection{Realizable dimension of $\mathbb{Z}$-labelled graphs}
We say that a $\mathbb{Z}$-labelled graph is \emph{$d$-realizable} if its lift graph is $d$-realizable, and the \emph{realizable dimension} $\rd(\hat{G},z)$ of a $\mathbb{Z}$-labelled graph $(\hat{G},z)$ is defined to be the realizable dimension of its lift graph. We verify the minor monotonicity of the parameter.
\begin{prop} \label{prop:minor}
    If a $\mathbb{Z}$-labelled graph $(\hat{G},z)$ is a minor of a $\mathbb{Z}$-labelled graph $(\hat{H},w)$, then $\rd(\hat{G},z) \leq \rd(\hat{H},w)$.
\end{prop}
\begin{proof}
    It suffices to show the statement in two cases: (i) $(\hat{G},z)$ is a subgraph of $(\hat{H},w)$, (ii) $(\hat{G},z)=(\hat{H}/e,w)$ for an edge $e \in \hat{E}(\hat{H})$.
    Let $G$ (resp. $H$) be the lift graph of $(\hat{G},z)$ (resp. $(\hat{H},w)$).
    
    For (i), suppose that $(\hat{G},z)$ is a subgraph of $(\hat{H},w)$.
    Let $(G,p)$ be a periodic framework in some dimension. Take a periodic framework $(H,q)$ in the same dimension such that $q|_{V(G)}=p$. By definition, $(H,q)$ admits an equivalent framework $(H,q')$ in $\mathbb{R}^d$, where $d=\rd(H)$. Then $(G,q'|_{V(G)})$ is equivalent to $(G,p)$. Hence $G$ is $d$-realizable. Thus the statement holds in the case (i). 
    
    For (ii), suppose that $(\hat{G},z)=(\hat{H}/e,w)$ for an edge $e=(i,j;0) \in \hat{E}(\hat{H})$. Let $\hat{V} \subset V(H)$ be a representative vertex set such that the quotient $\mathbb{Z}$-labelled graph of $H$ with respect to $\hat{V}$ is $(\hat{H},w)$. Let $(G,p)$ be a periodic framework in some dimension. Take a periodic framework $(H,q)$ in the same dimension such that $q|_{V(G)}=p$ and $q(i)=q(j)$. By definition, $(H,q)$ admits an equivalent framework $(H,q')$ in $\mathbb{R}^d$ with $d=\rd(H)$. Then, by the edge length constraint on $e$, we have $q'(i)=q'(j)$. Then equalities on the other edge lengths in $\hat{E}(\hat{H})$ implies that $(G,q'|_{V(G)})$ is equivalent to $(G,p)$. Hence $G$ is $d$-realizable. Thus the statment holds in the case (ii).    
\end{proof}
By Proposition~\ref{prop:minor}, the set of $d$-realizable $\mathbb{Z}$-labelled graphs is minor closed. 
We now state our main result in terms of $\mathbb{Z}$-symmetric graphs. 
Let $(K_k,0)$ be the balanced complete graph on $k$ vertices.
Let $K_3^{\bullet \bullet}$ be a directed multigraph on $\{1,2,3\}$ whose edge set consists of a single edge from $1$ to $2$ and parallel edges between $\{1,3\}$ and $\{2,3\}$.
\begin{theorem} \label{thm:main}
    For a $\mathbb{Z}$-symmetric graph $G$, the followings hold:
    \begin{enumerate}
        \item[(i)] $\rd(G) \leq 1$ if and only if $G$ has no minor isomorphic to the lift graph of $(K_3,0)$ nor to the lift graph of $(K_2^\bullet,z)$ for any simple labelling $z$.
        \item[(ii)] $\rd(G) \leq 2$ if and only if $G$ has no minor isomorphic to the lift graph of $(K_4,0)$ nor to the lift graph of $(K_3^{\bullet \bullet},z)$ for any simple labelling $z$.
    \end{enumerate}
\end{theorem}
The proof of Theorem~\ref{thm:main} is given in Section~\ref{sec:characterization}.
We remark that a minor closed class of $\mathbb{Z}$-labelled graphs is not necessarily characterized by a finite number of forbidden minors.
We should remark that our definition of minors for $\mathbb{Z}$-labelled graphs is consistent with that for $\Gamma$-labelled graphs for a general group $\Gamma$~\cite{geelen2009excluding}.
When the underlying group $\Gamma$ is finite, every minor closed class of $\Gamma$-labelled graphs is characterized by a finite number of forbidden minors \cite{geelen2009excluding}.

\subsection{Quotient frameworks}
For a periodic framework $(G,p)$, let $\hat{V} \subset V(G)$ be a representative vertex set, $(\hat{G},z)$ be the quotient $\mathbb{Z}$-labelled graph of $G$ with respect to $\hat{V}$, $\hat{p}=p|_{\hat{V}}$ be the positions of representative vertices, and $\ell$ be the lattice vector.
The quadruple $(\hat{G},z,\hat{p},\ell)$ is called the \emph{quotient framework} of $(G,p)$ with respect to $\hat{V}$.
We remark that a quotient framework depends on the choice of representative vertices.

We write $\aff(Y)$ for the affine span of a subset $Y$ of $\mathbb{R}^d$ and $\spa (Y)$ (or $\spa Y$) for its linear span.
We usually denote the number of vertices of a $\mathbb{Z}$-labelled graph $(\hat{G},z)$ as $n$ and denote $\hat{V}(\hat{G})=\{1, \ldots, n\}$.
A map $\hat{p}:\hat{V}(\hat{G}) \rightarrow \mathbb{R}^d$ is identified with a matrix $\hat{P} \in \mathbb{R}^{d \times n}$ whose $i$th column is $\hat{p}(i)$.

For a periodic framework $(G,p)$ in $\mathbb{R}^d$, let $(\hat{G},z,\hat{p},\ell)$ be the quotient framework of $(G,p)$ with respect to $\hat{V}$.
Then we have $\aff(p(V)) = \aff(\hat{p}(\hat{V}))+\spa\{\ell\}$ and the affine dimension of $(G,p)$ equals to
\[
    \rank \begin{pmatrix} \hat{P} & \ell \\ \bm{1}_n^\top & 0 \end{pmatrix} -1,
\]
where $\bm{1}_n \in \mathbb{R}^n$ is the all one vector.
In particular, the affine dimension of a periodic framework $(G,p)$ is at most $|\hat{V}|=|V/\mathbb{Z}|$.
Thus, the realizable dimension $\rd(G)$ of a $\mathbb{Z}$-symmetric graph $G$ is at most $|V/\mathbb{Z}|$.

\section{Lower Dimensional Affine Images and the Conic Condition}
Since the seminal work by Connelly~\cite{connelly1982rigidity,connelly2012tensegrities,connelly2018affine}, affine transformations have played an important role in the analysis of global or universal rigidity of frameworks. In this section we adapt known facts on affine transformations of frameworks to the periodic setting.
Let $\mathcal{S}^n$ be the set of $n \times n$ real symmetric matrices.
For $A, B \in \mathcal{S}^n$, let $\langle A,B \rangle = \tr(A^\top B)$ be the trace inner product.
If a symmetric matrix $A$ is positive semidefinite, we denote $A\succeq O$.

For a framework $(G,p)$ in $\mathbb{R}^d$, a matrix $A \in \mathbb{R}^{d \times d}$, and a vector $t \in \mathbb{R}^d$, define an \emph{affine image} $q:V(G) \rightarrow \mathbb{R}^d$ of $p$ by $q(i)=Ap(i)+t$ for all $i \in V(G)$. Then if $(G,p)$ and $(G,q)$ are equivalent, the equality of the length of $uv \in E(G)$ gives
\[
    (p(u)-p(v))^\top (A^\top A -I_d) (p(u)-p(v)) =0.
\]
In view of this, we say that an $\ell$-periodic framework $(G,p)$ in $\mathbb{R}^d$ satisfies the \emph{conic condition} if the edge directions including the lattice vector do not lie on a conic at infinity, i.e., if there is no nonzero matrix $S \in \mathcal{S}^d$ satisfying
\begin{equation} \label{eq:conic}
    (p(u)-p(v))^\top S (p(u)-p(v)) =0 \quad \text{ for all } uv \in E(G) \quad \text{ and } \quad \ell^\top S\ell=0.
\end{equation}

For example, for a periodic framework in $\mathbb{R}^d$, if there are two hyperplanes $H_1, H_2 \subset \mathbb{R}^d$ such that every edge including lattice vector is parallel to either $H_1$ or $H_2$, it does not satisfy the conic condition, since $S=a_1a_2^\top+a_2a_1^\top$ satisfies (\ref{eq:conic}), where $a_i$ is a normal vector of $H_i$ for $i=1,2$.
This in particular implies that, if $(G,p)$ satisfies the conic condition, then it is affine full dimensional.

The following proposition implies that, if $(G,p)$ does not satisfy the conic condition, then not only it admits an equivalent affine transformation but also has a lower affine-dimensional equivalent realization. The corresponding statement for ordinary finite frameworks has been already discussed in \cite{connelly2012tensegrities}. We check that the same argument works in the periodic setting for completeness.
\begin{prop} \label{prop:flex}
    Let $(G,p)$ be a periodic framework in $\mathbb{R}^d$ which does not satisfy the conic condition.
    Then there exists an affine $d'$-dimensional periodic framework $(G,q)$ equivalent to it for some $d' <d$.
\end{prop}
\begin{proof}
    Let $S \in \mathcal{S}^d$ be a nonzero matrix satisfying (\ref{eq:conic}).
    There is a real number $t \in \mathbb{R}$ such that $I_d -tS \succeq O$ and $d':=\rank (I_d -tS) \leq d-1$.
    Let $A\in\mathcal{S}^d$ be a singular matrix satisfying $A^2=I_d -tS$, and define $q:V(G) \rightarrow \mathbb{R}^d$ by $q(i)=Ap(i)$ for all $i \in V(G)$.
    Then $(G,q)$ is equivalent to $(G,p)$ and its affine dimension is at most $d'$.
    As $(G,q)$ is an affine image of $(G,p)$, if $(G,p)$ is periodic, $(G,q)$ is also periodic.
\end{proof}

\section{Periodic Super Stability}
In this section we introduce the concept of perodic super stability to prove the lower bound of realizable dimension.
Let $\bm{e}_i \in \mathbb{R}^n$ be the $i$-th unit vector, and $E_{i,j}=\bm{e}_i\bm{e}_j^\top \in \mathbb{R}^{n \times n}$. 
Let $\bm{1}_{n+1}' := \begin{pmatrix}\bm{1}_n^\top & 0 \end{pmatrix}^\top \in \mathbb{R}^{n+1}$.
Let $\mathcal{L}^{n+1}:=\{A \in \mathcal{S}^{n+1}:A\bm{1}_{n+1}'=0 \}$ and $\mathcal{L}^{n+1}_+:=\{A \in \mathcal{L}^{n+1}:A \succeq O\}$.

\subsection{Equilibrium stress and periodic super stability}
Recall that for a quotient $\mathbb{Z}$-labelled graph $(\hat{G},z)$ of a $\mathbb{Z}$-symmetric graph $G$, $(E \cup E_L)/\mathbb{Z}$ and $\hat{E}^\circ=\hat{E}\cup \{e_L\}$ are identified.
A map $\omega: (E \cup E_L)/\mathbb{Z} \rightarrow \mathbb{R}$ can be considered as an edge weight on $\hat{E}^\circ$.

For a periodic framework $(G,p)$ in $\mathbb{R}^d$, fix a representative vertex set $\hat{V} \subset V$, and let $(\hat{G},z,\hat{p},\ell)$ be the corresponding quotient framework. 
For an edge $e=(i,j;z(e)) \in \hat{E}$, define the indicator vector $i_e \in \mathbb{R}^{n+1}$ and the edge vector $v_e \in \mathbb{R}^d$ by
\begin{align*}
    &i_e := \begin{pmatrix} \bm{e}_j-\bm{e}_i \\ z(e) \end{pmatrix}, \\
    &v_e:=\begin{pmatrix} \hat{P} & l \end{pmatrix}i_e =\hat{p}(j)+z(e)l-\hat{p}(i),
\end{align*}
and for $e_L$, let $i_{e_L}:=\bm{e}_{n+1}$ and $v_{e_L}:=\ell$.
Borcea and Streinu~\cite{borcea2010periodic} defined a rigidity matrix $R(\hat{G},z,\hat{p},\ell)$ of a periodic framework by a $|\hat{E}^\circ| \times d(n+1)$ matrix whose row indexed by $e=(i,j;z(e)) \in E$ is 
\[
    \kbordermatrix{
        &&i&&j&&L \\
        &\cdots 0 \cdots&-v_e^\top&\cdots0\cdots&v_e^\top&\cdots 0 \cdots&z(e)v_e^\top
    },
\]
and a row indexed by $e_L$ is
\[
    \kbordermatrix{
        &&&&L \\
        &\cdots 0 \cdots&0&\cdots 0 \cdots&v_{e_L}^\top
    },
\]
where each vertex has associated consecutive $d$ columns and the last $d$ columns are associated with the lattice.
Note that if an edge $e$ is inverted, the labelling becomes $-z(e)$ and the edge vector becomes $-v_e$, so $R(\hat{G},z,\hat{p},\ell)$ is invariant by edge inversion.
We have the following lemma.
\begin{lemma} \label{lem:eq}
    Let $(G,p)$ be a periodic framework and let $\hat{V}, \hat{V}' \subset V(G)$ be representative vertex sets. 
    Let $(\hat{G},z,\hat{p},\ell)$ (resp. $(\hat{G}',z',\hat{p}',\ell)$) be the quotient framework of $(G,p)$ with respect to $\hat{V}$ (resp. $\hat{V}'$).
    Then for any $\omega : (E \cup E_L)/\mathbb{Z} \rightarrow \mathbb{R}$, $\omega \in \ker R(\hat{G},z,\hat{p},\ell)^\top$ if and only if $\omega \in \ker R(\hat{G}',z',\hat{p}',\ell)^\top$.
\end{lemma}
\begin{proof}
We already checked that $R(\hat{G},z,\hat{p},\ell)$ is invariant by edge inversion.
It suffices to prove the statement when $i \in \hat{V}$ is replaced by $i+\gamma \in \hat{V}'$ and other representative vertices are the same.
Observe that $\omega$ is in the left kernel of $R(\hat{G},z,\hat{p},\ell)$ if and only if
\begin{align}
    &-\sum_{e \in \delta_{\text{out}}(i)} \omega(e) v_e + \sum_{e \in \delta_{\text{in}}(i)} \omega(e) v_e = 0&  (i \in \hat{V}), \label{eq:eq1} \\
    &\sum_{e\in \hat{E}^\circ} \omega(e)z(e)v_e = 0, \label{eq:eq2}
\end{align}
where $\delta_{\text{out}}(i)$ (resp. $\delta_{\text{in}}(i)$) is the set of edges outgoing from (resp. incoming to) $i$ in $\hat{G}$ (we assume that selfloops incident to $i$ are not contained in $\delta_{\text{out}}(i)$ nor $\delta_{\text{in}}(i)$). Suppose that $\omega$ satisfies (\ref{eq:eq1}) and (\ref{eq:eq2}).
As $v_e$ is independent of the choice of representative vertices, (\ref{eq:eq1}) remains to hold.
(\ref{eq:eq2}) for $(\hat{G}',z',\hat{p}',\ell)$ becomes
\begin{align*}
    &\sum_{e\in \hat{E}^\circ -\delta_{\text{out}}(i)-\delta_{\text{in}}(i)} \omega(e)z(e)v_e + \sum_{e\in \delta_{\text{out}}(i)} \omega(e)(z(e)-\gamma)v_e + \sum_{e\in \delta_{\text{in}}(i)} \omega(e)(z(e)+\gamma)v_e \\
    =&\sum_{e\in \hat{E}^\circ} \omega(e)z(e)v_e + \gamma \left(-\sum_{e \in \delta_{\text{out}}(i)} \omega(e) v_e + \sum_{e \in \delta_{\text{in}}(i)} \omega(e) v_e\right),
\end{align*}
which is $0$, so (\ref{eq:eq2}) remains to hold. Thus we have confirmed the statement.
\end{proof}
Motivated by Lemma~\ref{lem:eq}, a map $\omega:(E \cup E_L)/\mathbb{Z} \rightarrow \mathbb{R}$ is called an \emph{equilibrium stress} of $(G,p)$ if, for some choice of representative vertices, it is in the left kernel of $R(\hat{G},z,\hat{p},\ell)$ when considered as a vector in $\mathbb{R}^{\hat{E}^\circ}$.
We remark that the equation (\ref{eq:eq1}) corresponds to an ordinary equilibrium condition if we interpret it in the periodic framework, and the equation (\ref{eq:eq2}) seems a new relation which is special in the periodic setting.

Let $(G,p)$ be a periodic framework and $(\hat{G},z)$ be a quotient $\mathbb{Z}$-labelled graph of $G$.
For an edge $e \in \hat{E}^\circ$, let $F_e:=i_ei_e^\top \in \mathcal{S}^{n+1}$. As $i_e$ is orthogonal to $\bm{1}_{n+1}'$, $F_e \in \mathcal{L}^{n+1}$. $F_e$ is invariant by edge inversion.
For an equilibrium stress $\omega:(E \cup E_L)/\mathbb{Z} \rightarrow \mathbb{R}$ of $(G,p)$, $L_{\hat{G},z,\omega}:=\sum_{e \in \hat{E}^\circ} \omega(e)F_e$ is called an \emph{equilibrium stress matrix} of $(G,p)$.
Although equilibrium stress matrix changes by switching, the following lemma shows that the signature of it is invariant.
\begin{lemma}
    Let $(\hat{G},z)$, $(\hat{G}',z')$ be quotient $\mathbb{Z}$-labelled graphs of a $\mathbb{Z}$-symmetric graph $G$.
    Then for $\omega:(E \cup E_L)/\mathbb{Z} \rightarrow \mathbb{R}$, the signatures of $L_{\hat{G},z,\omega}$ and $L_{\hat{G}',z',\omega}$ coincide.
\end{lemma}
\begin{proof}
We have checked that $L_{\hat{G},z,\omega}$ is invariant by edge inversion.
Define the incidence matrix $I(\hat{G},z) \in \mathbb{R}^{|\hat{E}^\circ| \times (n+1)}$ of $(\hat{G},z)$ as a matrix whose row indexed by $e \in \hat{E}^\circ$ is $i_e^\top$.
Then we have $L_{\hat{G},z,\omega}=I(\hat{G},z)^\top\diag(\omega)I(\hat{G},z)$.
In terms of an incidence matrix, a switching of $i$ by $\gamma$ results in adding $\gamma$ times $i$th column of incidence matrix to the last column.
Hence, by Sylvester's law of inertia, the signature is invariant by switching.
\end{proof}
We define the \emph{stress signature} $(n_+,n_-,n_0)$ of an equilibrium stress $\omega:(E\cup E_L)/\mathbb{Z} \rightarrow \mathbb{R}$ of a periodic framework $(G,p)$ to be the common signature of $L_{\hat{G},z,\omega}$, where $n_+$ (resp. $n_-$, $n_0$) is the number of positive (resp. negative, zero) eigenvalues of $L_{\hat{G},z,\omega}$.
Since $\omega^\top R(\hat{G},z,\hat{p},\ell)=0$ if and only if $\begin{pmatrix}\hat{P} & \ell \end{pmatrix}L_{\hat{G},z,\omega}=0$, $\bm{1}_{n+1}'$ and the row vectors of $\begin{pmatrix}\hat{P} & \ell \end{pmatrix}$ are always in the kernel of $L_{\hat{G},z,\omega}$. So $n_0$ of the stress signature of an equilibrium stress of $d$-dimensional periodic framework is at least $d+1$.
A stress signature of an equilibrium stress of $d$-dimensional periodic framework is \emph{nonnegative} if $n_-=0$, and it is \emph{full} if $n_0=d+1$.
We are now ready to define periodic super stability.
\begin{defi} \label{defi:ss}
    A periodic framework $(G,p)$ in $\mathbb{R}^d$ is \emph{periodically super stable} if the following conditions hold:
    \begin{itemize}
        \item[(i)] It has an equilibrium stress whose stress signature is nonnegative and full.
        \item[(ii)] $(G,p)$ satisfies the conic condition.
    \end{itemize}
\end{defi}
Note that by (ii), a periodically super stable framework is affine full dimensional.

\begin{eg}
Consider the $2$-dimensional periodic framework $(G,p)$ in Figure~\ref{fig:stress} (a).
If $\hat{V}$ is the set of white vertices in Figure~\ref{fig:stress} (a), Figure~\ref{fig:stress} (b) is the quotient graph of $G$ with respect to $\hat{V}$ with the additional selfloop $\{e_L\}$.
An equilibrium stress $\omega:(E \cup E_L)/\mathbb{Z} \rightarrow \mathbb{R}$ of $(G,p)$ is shown in Figure~\ref{fig:stress} (c), in which a number written on each edge shows a weight on the edge orbit it belongs to.
Then the equilibrium stress matrix is 
\[
    \kbordermatrix{%
    &1&2&3&L\\
    1& 1 & 1 & -2 & 1 \\
    2& 1 & 1 & -2 & 1 \\
    3& -2&-2 &  4 &-2 \\
    L& 1 & 1 & -2 & 1 \\
    },
\]
which is positive semidefinite with rank $1$. Hence the stress signature is $(n_+,n_-,n_0)=(1,0,3)$, which is nonnegative and full. As $(G,p)$ has at least three edge directions, it satisfies the conic condition. Hence the periodic framework $(G,p)$ in Figure~\ref{fig:stress} (a) is periodically super stable.
\begin{figure}
\centering
\begin{tikzpicture}
    \tikzset{vertex/.style={draw,fill=black,circle,inner sep=2pt,minimum size=3pt}}
    \tikzset{wvertex/.style={draw,fill=white,circle,inner sep=2pt,minimum size=3pt}}
    \node[vertex] (a1) at (0,0) {};
    \node[wvertex] (a2) at (4,0) {};
    \node[vertex] (a3) at (8,0) {};
    \node[vertex] (b1) at (0,2) {};
    \node[wvertex] (b2) at (4,2) {};
    \node[vertex] (b3) at (8,2) {};
    \node[vertex] (c1) at (2,1) {};
    \node[wvertex] (c2) at (6,1) {};
    \node (x1) at (-1,0) {};
    \node (x2) at (-1,0.5) {};
    \node (x3) at (-1,1) {};
    \node (x4) at (-1,1.5) {};
    \node (x5) at (-1,2) {};
    \node (y1) at (9,0) {};
    \node (y2) at (9,0.5) {};
    \node (y3) at (9,1) {};
    \node (y4) at (9,1.5) {};
    \node (y5) at (9,2) {};  
    \draw (a1) to (b1); \draw (a1) to (c1); \draw (b1) to (c1); \draw (a2) to (b2); \draw (a2) to (c2); \draw (b2) to (c2); \draw (a3) to (b3);
    \draw (c1) to (a2); \draw (c1) to (b2); \draw (c2) to (a3); \draw (c2) to (b3);
    \draw (a1) to (x2); \draw (b1) to (x4);
    \draw (a3) to (y2); \draw (b3) to (y4); 
    \node (cap) at (3,-0.5) {(a)};
    
    \node (x0) at (-3,0) {};
    \node (x1) at (-1,0) {}; \node (name_x) at (-0.8,0) {$x$};
    \node (y0) at (-2.5,-0.5) {};
    \node (y1) at (-2.5,1.5) {}; \node (name_y) at (-2.5,1.7) {$y$};
    \draw[->,>=latex] (x0) to (x1);
    \draw[->,>=latex] (y0) to (y1);
\end{tikzpicture}
\\
\begin{tikzpicture}
    \tikzset{vertex/.style={draw,fill=black,circle,inner sep=2pt,minimum size=3pt}}
    \node[vertex] (a) at (0,0) {}; \node (name_a) at (-0.3,0) {$1$};
    \node[vertex] (b) at (0,2) {}; \node (name_b) at (-0.3,2) {$2$};
    \node[vertex] (c) at (2,1) {}; \node (name_c) at (2.3,1) {$3$};
    \draw[->,>=latex] (a) to node [midway,left] {$+0$} (b);
    \draw[->,>=latex] (c) to [bend left=20] node [midway,below] {$+1$} (a);
    \draw[->,>=latex] (c) to [bend right=20] node [midway,below] {$+0$} (a);
    \draw[->,>=latex] (c) to [bend right=20] node [midway,above] {$+1$} (b);
    \draw[->,>=latex] (c) to [bend left=20] node [midway,above] {$+0$} (b);
    \node (cap) at (1,-0.5) {(b)};

    \node[vertex] (l) at (1.8,2.5) {}; \node (name_L) at (1.8,2.2) {$L$};
    \draw[->,>=latex] (l) to [loop above] node [midway,right] {$+1$} node [midway,left] {$e_L$} (l);
\end{tikzpicture}
\qquad
\begin{tikzpicture}
    \tikzset{vertex/.style={draw,fill=black,circle,inner sep=2pt,minimum size=3pt}}
    \node[vertex] (a1) at (0,0) {};
    \node[vertex] (a2) at (4,0) {};
    \node[vertex] (a3) at (8,0) {};
    \node[vertex] (b1) at (0,2) {};
    \node[vertex] (b2) at (4,2) {};
    \node[vertex] (b3) at (8,2) {};
    \node[vertex] (c1) at (2,1) {};
    \node[vertex] (c2) at (6,1) {};
    \node (x1) at (-1,0) {};
    \node (x2) at (-1,0.5) {};
    \node (x3) at (-1,1) {};
    \node (x4) at (-1,1.5) {};
    \node (x5) at (-1,2) {};
    \node (y1) at (9,0) {};
    \node (y2) at (9,0.5) {};
    \node (y3) at (9,1) {};
    \node (y4) at (9,1.5) {};
    \node (y5) at (9,2) {};  
    \draw (a1) to node [midway,right] {$-1$} (b1); \draw (a1) to node [midway,below] {$1$} (c1); \draw (b1) to node [midway,above] {$1$} (c1); \draw (a2) to node [midway,right] {$-1$} (b2); \draw (a2) to node [midway,below] {$1$} (c2); \draw (b2) to node [midway,above] {$1$} (c2); \draw (a3) to node [midway,right] {$-1$} (b3);
    \draw (c1) to node [midway,above] {$1$} (a2); \draw (c1) to node [midway,below] {$1$} (b2); \draw (c2) to node [midway,above] {$1$} (a3); \draw (c2) to node [midway,below] {$1$} (b3);
    \draw (a1) to (x2); \draw (b1) to (x4);
    \draw (a3) to (y2); \draw (b3) to (y4); 
    \node (cap) at (4,-0.5) {(c)};

    \node (name_E0) at (-1.3,2.5) {$E_L$};
    \node (l1) at (-1,2.5) {};
    \node[vertex] (l2) at (0,2.5) {};
    \node[vertex] (l3) at (4,2.5) {};
    \node[vertex] (l4) at (8,2.5) {};
    \node (l5) at (9,2.5) {};
    \draw (l1) to (l2);
    \draw (l2) to node [midway,above] {$-1$} (l3);
    \draw (l3) to node [midway,above] {$-1$} (l4);
    \draw (l4) to (l5);
\end{tikzpicture}
\caption{(a) A $2$-dimensional periodic framework $(G,p)$ with a representative vertex set $\hat{V}$ denoted by white circles. (b) The quotient graph of $G$ with respect to $\hat{V}$ with the additional selfloop $\{e_L\}$. (c) An equilibrium stress $\omega$ of $(G,p)$.}.
\label{fig:stress}
\end{figure}
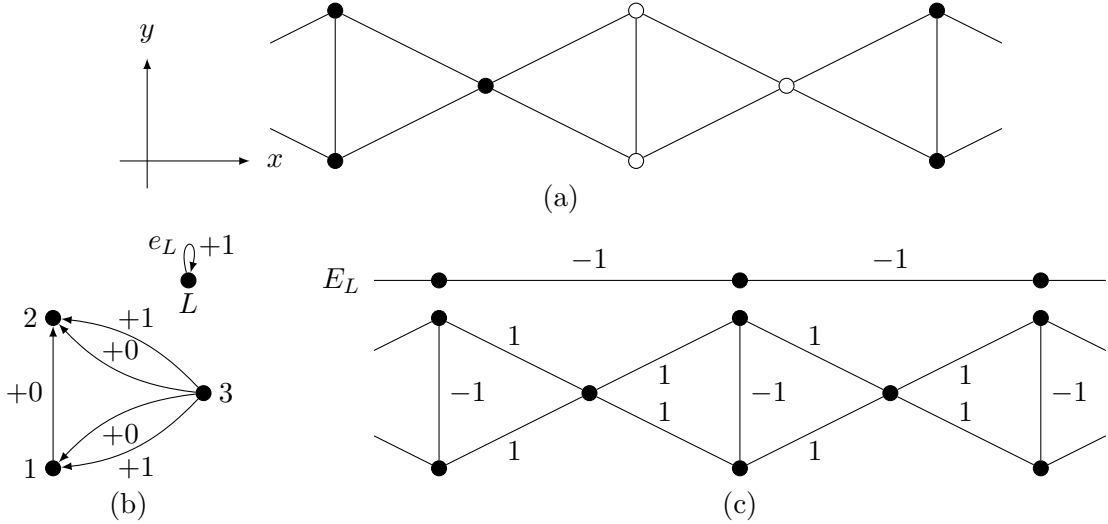
\end{eg}
A periodic framework $(G,p)$ is \emph{periodically universally rigid} if every periodic framework $(G,q)$ in any dimension equivalent to $(G,p)$ is congruent to $(G,p)$. 
We have the following periodic analogue of~\cite{connelly1982rigidity}. A variant of the following proposition for quotient frameworks has been shown in an unpublished manuscript~\cite{sean}.
\begin{prop} \label{prop:superstable}
    A periodically super stable framework is periodically universally rigid.
\end{prop}

\section{Obstacles}
For a directed multigraph $D$ and two vertices $i,j$ of $D$, the \emph{multiplicity} between $i$ and $j$ is the number of edges between $i$ and $j$ when the edge direction is forgotten.
For a directed multigraph $D$, let $\si(D)$ be the simple loopless undirected graph obtained from $D$ by removing selfloops and forgetting the direction and the multiplicity.
\subsection{Relationship with finite frameworks} \label{subseq:finite}
The first observation for constructing obstacles is done by comparing the realizable dimension of a $\mathbb{Z}$-labelled graph $(\hat{G},z)$ and the realizable dimension of a finite graph $\si(\hat{G})$.
We have the following.
\begin{prop} \label{prop:finite}
    For a $\mathbb{Z}$-labelled graph $(\hat{G},z)$, $\rd(\hat{G},z) \geq \rd(\si(\hat{G}))$.
\end{prop}

The idea of the proof of Proposition~\ref{prop:finite} is simple. If $(\si(\hat{G}),p)$ is a framework in some dimension, we hypothetically regard this as a periodic framework $(G,p)$ with zero lattice vector, where $G$ is a lift graph of $(\hat{G},z)$, and use the realizability of $G$ to find a lower dimensional finite framework equivalent to $(\si(\hat{G}),p)$. We make this argument more rigorous. 

Let $(\hat{G},z)$ be a $\mathbb{Z}$-labelled graph. 
For each nonnegative integer $d$, consider a map $f_d:\mathbb{R}^{dn}\times\mathbb{R}^{d} \rightarrow \mathbb{R}^{\hat{E}^\circ}$ defined by $f_d(\hat{p},\ell):=(\|v_e\|^2)_{e \in \hat{E}^\circ}$. Let $\mathbb{R}^d_{\neq 0}:=\mathbb{R}^d\setminus \{0\}$.
As the lattice vector of a periodic framework is a nonzero vector, we have 
\begin{equation*}
\{(\hat{p},\ell):(\hat{G},z,\hat{p},\ell) \text{ is a quotient framework of a periodic framework in } \mathbb{R}^d\}=\mathbb{R}^{dn}\times\mathbb{R}^{d}_{\neq 0}.
\end{equation*}
Then $f_0(\mathbb{R}^{0}\times\mathbb{R}^{0}_{\neq 0}) \subseteq f_1 (\mathbb{R}^{n}\times\mathbb{R}^{1}_{\neq 0}) \subseteq \cdots$ is an ascending chain of subsets of $\mathbb{R}^{\hat{E}^\circ}$.
By definition, $(\hat{G},z)$ is $d$-realizable if and only if $f_d (\mathbb{R}^{dn}\times\mathbb{R}^{d}_{\neq 0}) = f_{d+1}(\mathbb{R}^{(d+1)n}\times\mathbb{R}^{d+1}_{\neq 0})=\cdots$ holds, and $\rd(\hat{G},z)$ equals to the minimum $d$ satisfying such relation.
The next lemma is a key. A similar argument can be found in \cite{belk2007realizability}.
\begin{lemma} \label{lem:closed}
    For every nonnegative integer $d$, $\Im f_{d}$ is closed.
\end{lemma}
\begin{proof}
    For brevity, we denote $f:=f_d$.
    Let $\{y_t\}_{t \in \mathbb{N}}$ be a sequence in $\Im f$ converging to $y$. We show that $y \in \Im f$.
    We claim that there exists a closed ball $B \subset \mathbb{R}^{dn}\times\mathbb{R}^{d}$ such that $B \cap f^{-1}(y_t) \neq \emptyset$ for all $t \in \mathbb{N}$.
    To see the claim, fix vertices $v_1,\ldots,v_k$ in each connected component of $\si(\hat{G})$. Note that for a vertex set $\hat{U}$ of a connected component of $\si(\hat{G})$, the value $f(\hat{p},\ell)$ is invariant under a translation of $\hat{p}|_{\hat{U}}$.
    Hence, there exists $(\hat{p},\ell) \in f^{-1}(y_t)$ such that $\hat{p}(v_1)=\cdots=\hat{p}(v_k)=0$.
    Then since $\|\ell\|$ is bounded and $\|\hat{p}(j)+z(e)\ell-\hat{p}(i)\|$ is also bounded for all $(i,j;e) \in \hat{E}$, one can deduce that $\|\hat{p}(j)\|$ is bounded for every $j \in \hat{V}$ by triangle inequality. Hence $(\hat{p},\ell)$ is contained in a ball whose radius is independent of $t$.
    Therefore $B \cap f^{-1}(y_t) \neq \emptyset$.

    The rest of the argument is immediate.
    For each $t \in \mathbb{N}$, take $(\hat{p}_t,\ell_t) \in f^{-1}(y_t) \cap B$. Then by the compactness of $B$, a sequence $\{(\hat{p}_t,\ell_t)\}_{t \in \mathbb{N}}$ has an accumulation point $x=(\hat{p}^*,\ell^*) \in B$. Then, by the continuity of $f$, we have $f(x)=y$. Thus $y \in \Im f$.
\end{proof}
We are now ready to prove Proposition~\ref{prop:finite}.
\begin{proof} [Proof of Proposition~\ref{prop:finite}]
    Suppose that $(\hat{G},z)$ is $d$-realizable.
    Then, for any nonnegative integer $d'$, we have
    \begin{equation} \label{eq:seq}
    f_{d'}(\mathbb{R}^{d'n}\times \mathbb{R}^{d'}_{\neq 0}) \subseteq f_d(\mathbb{R}^{dn} \times \mathbb{R}^d_{\neq 0}) \subseteq \Im f_d.
    \end{equation}
    Since the closure of $\mathbb{R}^{d'n}\times \mathbb{R}^{d'}_{\neq 0}$ is $\mathbb{R}^{d'n}\times \mathbb{R}^{d'}$, by the continuity of $f_{d'}$, the closure of $f_{d'}(\mathbb{R}^{d'n}\times \mathbb{R}^d_{\neq 0})$ includes $\Im f_{d'}$.
    Hence by Lemma \ref{lem:closed} and (\ref{eq:seq}), we get $\Im f_{d'} \subseteq \Im f_d$ for any $d'$.
    
    Let $(\si(\hat{G}),p)$ be a framework of a simple graph $\si(\hat{G})$ in $\mathbb{R}^{d'}$ for some $d'$. We show that there is an equivalent framework in $\mathbb{R}^d$.
    As $\Im f_{d'} \subseteq \Im f_d$, there exists $(q,\ell) \in \mathbb{R}^{dn}\times\mathbb{R}^{d}$ such that $f_d(q,\ell)=f_{d'}(p,0)$.
    Then $\|\ell\|^2=0$, i.e., $\ell=0$.
    Then the equiality $f_d(q,0)=f_{d'}(p,0)$ implies that $(\si(\hat{G}),q)$ is a framework in $\mathbb{R}^d$ equivalent to $(\si(\hat{G}),p)$.
    Thus, $\si(\hat{G})$ is $d$-realizable.
\end{proof}
Belk and Connelly~\cite{belk2007realizability} have shown that $\rd(K_n) = n-1$ and $\rd(K_{2,2,2}) \geq 4$.
Hence by Proposition~\ref{prop:finite}, we obtain the following obstacles in the periodic setting.
\begin{corollary} \label{cor:finite2}
    We have the followings.
    \begin{enumerate}
        \item[(i)] For any $\mathbb{Z}$-labelled graph $(K_n,z)$, $\rd(K_n,z) \geq n-1$.
        \item[(ii)] For any $\mathbb{Z}$-labelled graph $(K_{2,2,2},z)$, $\rd(K_{2,2,2},z) \geq 4$.
    \end{enumerate}
\end{corollary}

\subsection{Realizability of complete graphs}
We now provide obstacles of $d$-realizability by determining $\rd(\hat{G},z)$ for a $\mathbb{Z}$-labelled graph $(\hat{G},z)$ when $\si(\hat{G})$ is a complete graph.
For a directed multigraph $\hat{G}=(\hat{V},\hat{E})$, its \emph{multiplicity graph} is a simple undirected graph on $\hat{V}$ such that distinct vertices $i,j \in \hat{V}$ are adjacent if the multiplicity between $i$ and $j$ is at least $2$.
We have the following combinatorial statement of the linear matroid generated by $\{F_e:e \in \hat{E}^\circ\}$.
\begin{lemma} \label{lem:span}
    Let $(\hat{G},z)$ be a $\mathbb{Z}$-labelled graph. Let $\mathcal{F}(\hat{G},z):=\{F_e:e\in\hat{E}^\circ\} \subset \mathcal{L}^{n+1}$. We have the followings:
    \begin{itemize}
        \item[(i)] The set of matrices $\mathcal{F}(\hat{G},z)$ is linearly independent if and only if $\hat{G}$ has no selfloop, the multiplicity between any two vertices is at most two, and the multiplicity graph of $\hat{G}$ is a forest.
        \item[(ii)] The set of matrices $\mathcal{F}(\hat{G},z)$ spans the whole space $\mathcal{L}^{n+1}$ if and only if $\si(\hat{G})=K_n$ and the multiplicity graph of $\hat{G}$ is spanning.
    \end{itemize}
\end{lemma}
\begin{proof}
    Recall that $E_{i,j}$ denotes the square matrix of size $n+1$ whose $(i,j)$-th entry is one and the remainig entries are zero.
    For $1 \leq i<j\leq n$, define $A_{i,j}, B_{i,j} \in \mathcal{L}^{n+1}$ by $A_{i,j}=E_{i,i}+E_{j,j}-E_{i,j}-E_{j,i}$ and $B_{i,j}=E_{i,n+1}+E_{n+1,i}-E_{j,n+1}-E_{n+1,j}$.
    By the definition of $\mathcal{L}^{n+1}$, $\dim \mathcal{L}^{n+1}=\binom{n+1}{2}$, and $\{A_{i,j}:1 \leq i<j\leq n\} \cup \{B_{i,j}:\{i,j\}\in E(T)\} \cup \{E_{n+1,n+1}\}$ forms a basis of $\mathcal{L}^{n+1}$ for any spanning tree $T$ on $\{1,\ldots,n\}$. Note that $F_{e_L}=E_{n+1,n+1}$.

    Necessity of (i): Suppose that $\hat{G}$ contains a selfloop $e$, then $\{F_e, F_{e_L}\}$ is linearly dependent.
    Suppose that the multiplicity between $i$ and $j$ is at least three.
    If $e_1, e_2 \in \hat{E}$ are two parallel edges between $i$ and $j$, then we have
    \begin{equation} \label{eq:span1}
        \spa \{F_{e_L},F_{e_1},F_{e_2}\} = \spa \{A_{i,j},B_{i,j},E_{n+1,n+1}\}.
    \end{equation}
    If $e_3 \in \hat{E}$ is another edge between $i$ and $j$, then $F_{e_3}$ is in $\spa \{A_{i,j},B_{i,j},E_{n+1,n+1}\}$. So, by (\ref{eq:span1}), $\{F_{e_L}\}\cup\{F_{e_i}:i=1,2,3\}$ is linearly dependent.

    Suppose that the multiplicity graph contains a cycle on $i_1,\ldots,i_t \in \hat{V}$ ($t \geq 3$). Then $\hat{G}$ has parallel edges $e_k^1,e_k^2$ between $i_k$ and $i_{k+1}$ for $k=1,\ldots,t$ ($i_{t+1}=i_1$).
    Let $\mathcal{F}:=\{F_{e_L}\} \cup \{F_{e_j^i}:i=1,2, j=1,\ldots,t\}$. We show $\mathcal{F}$ is linearly dependent.
    By (\ref{eq:span1}), we have $B_{i_k,i_{k+1}} \in \spa \{F_{e_L},F_{e_k^1},F_{e_k^2}\} \setminus \spa\{F_{e_L}\}$ for each $k=1,\ldots,t$. Also by the definition of $B_{i,j}$, the identity $\sum_{k=1}^t B_{i_k,i_{k+1}}=0$ holds. Hence $\mathcal{F}$ are linearly dependent. Thus the necessity of (i) follows.

    Sufficiency of (i):
    For the sufficiency, by adding as many edges as possible, it suffices to prove the statement when $\si(\hat{G})=K_n$, $\hat{G}$ contains no selfloop, multiplicity between any two vertices is at most two, and the multiplicity graph is a spanning tree $T$ on $\{1,\ldots,n\}$. 
    Then $|\hat{E}^\circ|=|\hat{E}|+1=\binom{n}{2}+(n-1)+1=\binom{n+1}{2}=\dim \mathcal{L}^{n+1}$. So, it suffices to show that the linear subspace $\mathcal{K}$, a subspace spanned by $\{F_e:e\in\hat{E}^\circ\}$, contains a basis of $\mathcal{L}^{n+1}$.
    By (\ref{eq:span1}), $\mathcal{K}$ contains $A_{i,j}, B_{i,j}$ for $\{i,j\} \in E(T)$.
    
    For $\{i,j\} \not\in E(T)$, we claim that $\mathcal{K}$ contains $A_{i,j}$. To see this, let $i=i_1,\ldots,i_t=j$ be a path in the multiplicity graph, and let $e \in \hat{E}$ be an edge between $i$ and $j$.
    By (\ref{eq:span1}), $\mathcal{K}$ contains $A_{i_k,i_{k+1}}$ and $B_{i_k,i_{k+1}}$ for $k=1,\ldots, t-1$.
    Hence $B_{i_t,i_1}=-\sum_{k=1}^{t-1} B_{i_k,i_{k+1}}$ is contained in $\mathcal{K}$.
    Since $A_{i,j}=A_{i_t,i_1} \in \spa \{F_{e_L},F_{e},B_{i_t,i_1}\}$, $\mathcal{K}$ also contains $A_{i,j}$.

    Since $F_{e_L}=E_{n+1,n+1} \in \mathcal{K}$, $\mathcal{K}$ contains a basis $\{A_{i,j}:1 \leq i<j\leq n\} \cup \{B_{i,j}:\{i,j\}\in E(T)\} \cup \{E_{n+1,n+1}\}$. Thus the sufficiency of (i) follows.

    Necessity of (ii): Suppose that $\mathcal{F}(\hat{G},z)$ spans $\mathcal{L}^{n+1}$. Then there exists a subset $\hat{E}' \subseteq \hat{E}(\hat{G})$ containing $\{e_L\}$ such that $\{F_e:e \in \hat{E}'\}$ is a basis of $\mathcal{L}^{n+1}$. Define the $\mathbb{Z}$-labelled subgraph $(\hat{H},w)$ of $(\hat{G},z)$ by $\hat{V}(\hat{H})=\hat{V}(\hat{G})$, $\hat{E}(\hat{H})=\hat{E}'$, and $w:=z|_{\hat{E}'}$. By (i), $\hat{H}$ has no selfloop and multiplicity between any two vertices is at most two and the multiplicity graph of $\hat{H}$ is a forest.
    Since $|\hat{E}(\hat{H})^\circ|=\binom{n+1}{2}$, this implies that $\si(\hat{H})=K_n$ and the multiplicity graph of $\hat{H}$ is a spanning tree. Hence $\si(\hat{G})=K_n$ and the multiplicity graph of $\hat{G}$ is spanning.
    
    Sufficiency of (ii): Suppose that $\si(\hat{G})=K_n$ and the multiplicity graph of $\hat{G}$ is spanning. Then there is a $\mathbb{Z}$-labelled subgraph $(\hat{H},w)$ of $(\hat{G},z)$ such that $\si(\hat{H})=K_n$ and $\hat{H}$ contains no selfloop and multiplicity between any two vertices is at most two and the multiplicity graph is a spanning tree $T$ on $\{1,\ldots,n\}$. Then $|\hat{E}(\hat{H})^\circ|=\binom{n+1}{2}$, so $\mathcal{F}(\hat{H},w)$ is a basis of $\mathcal{L}^{n+1}$ by (i). Thus $\mathcal{F}(\hat{G},z)$ contains a basis of $\mathcal{L}^{n+1}$. Hence $\mathcal{F}(\hat{G},z)$ is spanning.

    
    
    
\end{proof}
We are now ready to determine the realizable dimension of $\mathbb{Z}$-labelled graphs $(\hat{G},z)$ satisfying $\si(\hat{G})=K_n$.
\begin{prop} \label{prop:Kn}
    Let $G$ be a $\mathbb{Z}$-symmetric graph and $(\hat{G},z)$ be its quotient $\mathbb{Z}$-labelled graph satisfying $\si(\hat{G})=K_n$. We have the followings:
    \begin{itemize}
        \item[(i)] If the multiplicity graph of $\hat{G}$ is not spanning, $\rd(G) = n-1$.
        \item[(ii)] If the multiplicity graph of $\hat{G}$ is spanning, every periodic framework $(G,p)$ in any dimension is periodically universally rigid. In particular, $\rd(G) =n$.
    \end{itemize}
\end{prop}
\begin{proof}
    By $\si(\hat{G})=K_n$, Proposition~\ref{prop:finite}, and Corollary~\ref{cor:finite2} (i), we have $n-1=\rd(K_n) \leq \rd(G)$. Since the realizable dimension is at most the number of vertex orbits, we also have $\rd(G)\leq |V(G)/\mathbb{Z}|=n$.
    Let $\hat{V}\subset V$ be a representative vertex set such that the quotient $\mathbb{Z}$-labelled graph of $G$ with respect to $\hat{V}$ is $(\hat{G},z)$.
    
    For (i), assume that the multiplicity graph of $\hat{G}$ is not spanning. Let $(G,p)$ be an affine $n$-dimensional framework in $\mathbb{R}^n$. We show that $(G,p)$ does not satisfy the conic condition. To see this, let $\hat{p}=p|_{\hat{V}}$ and let $\ell$ be the lattice vector of $(G,p)$. Observe that for a symmetric matrix $S \in \mathcal{S}^n$, (\ref{eq:conic}) holds if and only if
    \begin{align} \label{eq:conic2}
        \left\langle S, \begin{pmatrix} \hat{P} & \ell \end{pmatrix} F_e \begin{pmatrix} \hat{P} & \ell \end{pmatrix}^\top \right\rangle =0
        & \quad \text{ for all } e \in \hat{E}^\circ
    \end{align}
    holds. By Lemma~\ref{lem:span} (ii), there are at most $\binom{n+1}{2}-1$ linearly independent matrices in $\{F_e:e\in\hat{E}^\circ\}$. Hence there is a nonzero matrix $S \in \mathcal{S}^{n}$ satisfying (\ref{eq:conic2}). Thus we have confirmed that the conic condition does not hold. Then by Proposition~\ref{prop:flex}, $(G,p)$ admits a $d'$-dimensional equivalent periodic framework for some $d'<n$.
    Hence we have $\rd(G) \leq n-1$, so $\rd(G)=n-1$.
    
    For (ii), suppose that the multiplicity graph of $\hat{G}$ is spanning and $(G,p)$ is a periodic framework in $\mathbb{R}^d$.
    Since every framework can be considered as an affine full-dimensional framework in some dimension, and a periodic framework is periodically universally rigid if and only if a periodic framework congruent to it is periodic universally rigid, we may assume that the affine dimension of $(G,p)$ is $d$.
    Let $a_1, \ldots, a_{n-d} \in \mathbb{R}^{n+1}$ be a basis of the kernel of
    \[
        \begin{pmatrix} \hat{P} & \ell \\ \bm{1}_n^\top &0  \end{pmatrix}.
    \]
    Since $a_i$ is orthogonal to $\bm{1}_{n+1}'$, the matrix $\Omega=\sum_{i=1}^{n-d} a_i a_i^\top$ is in $\mathcal{L}^{n+1}$. Since $\mathcal{L}^{n+1}$ is spanned by $\{F_e:e \in \hat{E}^\circ\}$ by Lemma~\ref{lem:span} (ii), there exists $\omega \in \mathbb{R}^{\hat{E}^\circ}$ such that $\Omega = \sum_{e \in \hat{E}^\circ} \omega(e)F_e$. Then $\omega$ is an equilibrium stress of $(G,p)$ and $\Omega=L_{\hat{G},z,\omega}$. Since $\rank \Omega = n-d$ and $\Omega \succeq O$, the stress signature of $\omega$ is full and nonnegative, so it certifies the first condition of periodic super stability for $(G,p)$.
    It remains to check the conic condition. For this let $S \in \mathcal{S}^d$ be a matrix satisfying (\ref{eq:conic}). Then we have
    \[
        \left\langle \begin{pmatrix} \hat{P} & \ell \end{pmatrix}^\top S \begin{pmatrix} \hat{P} & \ell \end{pmatrix}, F_e \right\rangle =0
    \]
    for all $e \in \hat{E}^\circ$. By translation, we may assume that the center of gravity of $\hat{p}(\hat{V})$ is at the origin. Then $\begin{pmatrix} \hat{P} & \ell \end{pmatrix}^\top S \begin{pmatrix} \hat{P} & \ell \end{pmatrix} \in \mathcal{L}^{n+1}$ and it is orthogonal to matrices $\{F_e:e \in \hat{E}^\circ\}$, which spans $\mathcal{L}^{n+1}$ by Lemma~\ref{lem:span} (ii). Hence $\begin{pmatrix} \hat{P} & \ell \end{pmatrix}^\top S \begin{pmatrix} \hat{P} & \ell \end{pmatrix}=0$.
    As $\begin{pmatrix} \hat{P} & \ell \end{pmatrix}$ is row independent, this implies that $S=0$. Hence $(G,p)$ satisfies the conic condition.
    Thus we have confirmed that $(G,p)$ is periodically super stable.
    By Proposition~\ref{prop:superstable}, $(G,p)$ is periodically universally rigid.

    If $(G,p)$
    In particular, by considering the case $n=d$, we have $\rd(G)=n$.
\end{proof}
\section{Characterizations} \label{sec:characterization}
\subsection{Gluing operations}
We extend the $k$-sum operations of finite graphs to $\mathbb{Z}$-labelled graphs.

We denote a balanced simple complete graph of size $k\geq 0$ as $(K_k,0)$.
Suppose that $(\hat{G},z)$ is a $\mathbb{Z}$-labelled graph and $(\hat{G_1},z_1)$, $(\hat{G_2},z_2)$ are subgraphs of $(\hat{G},z)$ such that $\hat{V_1}\cup\hat{V_2}=\hat{V}$, $\hat{E_1}\cup\hat{E_2}=\hat{E}$, and $z_i=z|_{\hat{E_i}}$ for $i=1,2$. Then we say that $(\hat{G},z)$ is the \emph{union} of $(\hat{G_1},z_1)$ and $(\hat{G_2},z_2)$, and it is denoted as $(\hat{G_1},z_1) \cup (\hat{G_2},z_2)$.
The \emph{intersection} of $(\hat{G_1},z_1)$ and $(\hat{G_2},z_2)$ is defined as a $\mathbb{Z}$-labelled graph on $(\hat{V_1}\cap \hat{V_2}, \hat{E_1}\cap \hat{E_2})$ with the edge labelling $z_1|_{\hat{E_1}\cap \hat{E_2}}$, and it is denoted as $(\hat{G_1},z_1) \cap (\hat{G_2},z_2)$.
If $z_1=0$ and $(\hat{G_1},0) \cap (\hat{G_2},z_2)=(K_k,0)$, then $(\hat{G_1},0) \cup (\hat{G_2},z_2)$ is called the \emph{balanced $k$-sum} of $(\hat{G_1},0)$ and $(\hat{G_2},z_2)$.
We show that balanced $k$-sum preserves $d$-realizability.
\begin{lemma} \label{lem:bk-sum}
    Let $(\hat{G},z)$ be a balanced $k$-sum of $(\hat{G_1},0)$ and $(\hat{G_2},z_2)$ for some $k$. 
    Then $\rd(\hat{G},z) = \max\{\rd(\hat{G_1},0), \rd(\hat{G_2},z_2)\}$.
\end{lemma}
\begin{proof}
As $(\hat{G_1},0)$ and $(\hat{G_2},z_2)$ are subgraphs of $(\hat{G},z)$, minor monotonicity of $\rd(\cdot)$ implies $\rd(\hat{G},z) \geq \max\{\rd(\hat{G_1},0), \rd(\hat{G_2},z_2)\}$.

For the converse inequality, suppose that $(\hat{G_1},0)$ and $(\hat{G_2},z_2)$ are $d$-realizable.
Let $G=(V,E)$ be the lift graph of $(\hat{G},z)$ and let $\hat{V} \subset V$ be a representative vertex set such that the quotient $\mathbb{Z}$-labelled graph of $G$ with respect to $\hat{V}$ is $(\hat{G},z)$.
Let $G_i=(V_i,E_i)$ be the lift graph of $(\hat{G_1},z_i)$, and let $\hat{V}_i=V_i \cap \hat{V}$ for each $i=1,2$, where $z_1=0$. Note that $\hat{U}:=\hat{V}_1 \cap \hat{V}_2$ forms a clique of size $k$ in $G$.
Consider a periodic framework $(G,p)$ in some dimension. We show that $(G,p)$ admits an equivalent periodic framework in $\mathbb{R}^d$.

As both $G_1$, $G_2$ are $d$-realizable, there exists a periodic framework $(G_i,q_i)$ in $\mathbb{R}^d$ equivalent to $(G_i,p|_{V_i})$ for $i=1,2$. Let $\ell_i$ be the lattice vector of $(G_i,q_i)$ for $i=1,2$.
As $\hat{U}$ forms a clique in $G$, the finite frameworks $(K_{\hat{U}},q_1|_{\hat{U}})$ and $(K_{\hat{U}},q_2|_{\hat{U}})$ are congruent, where $K_{\hat{U}}$ is a complete graph on $\hat{U}$. Hence there is an orthogonal matrix $A \in O(\mathbb{R}^d)$ and $t \in \mathbb{R}^d$ such that $q_2(v)=Aq_1(v)+t$ for all $v \in \hat{U}$.
Define $q:V(G)\rightarrow \mathbb{R}^d$ by $q(v+\gamma):=Aq_1(v) + t + \gamma \ell_2$ for $v \in \hat{V}_1$ and $\gamma \in \mathbb{Z}$, and $q(v+\gamma):=q_2(v)+\gamma \ell_2$ for $v \in \hat{V}_2$ and $\gamma \in \mathbb{Z}$.
Then we claim that that $(G,p)$ is equivalent to $(G,q)$. For an edge in $E_2$, its length in $p$ equals to its length in $q_2$, so it equals to its length in $q$.
For an edge $e \in E_1$, as $(\hat{G_1},0)$ is balanced, there is an edge $e'$ in the same orbit as $e$ whose endvertices belong to $\hat{V}_1$. So it suffices to check the equality of an edge length of $e'$ in $p$ and $q$, which holds as $A$ is an orthogonal matrix.
Thus we have checked that $(G,p)$ is equivalent to $(G,q)$.
Hence the converse inequality follows.
\end{proof}
\begin{figure}
\centering
\begin{tikzpicture}
    \tikzset{vertex/.style={draw,fill=black,circle,inner sep=2pt,minimum size=3pt}}
    \node[vertex] (a) at (2,0) {}; \node (name_a) at (2.3,0) {$1$};
    \node[vertex] (b) at (0,1) {}; \node (name_b) at (-0.3,1) {$3$};
    \node[vertex] (c) at (2,2) {}; \node (name_c) at (2.3,2) {$2$};
    \draw[->,>=latex] (a) to node  [midway,below] {$+0$} (b);
    \draw[->,>=latex] (b) to [bend right=20] node [midway,below] {$+0$} (c);
    \draw[->,>=latex] (b) to [bend left=20] node [midway,above] {$+1$} (c);
    \draw[->,>=latex] (a) to node [midway,right] {$+0$} (c);
    \node (cap) at (1,-0.7) {(a)};
\end{tikzpicture}
\quad
\begin{tikzpicture}
    \tikzset{vertex/.style={draw,fill=black,circle,inner sep=2pt,minimum size=3pt}}
    \node[vertex] (a) at (2,0) {}; \node (name_a) at (1.7,0) {$1$};
    \node[vertex] (d) at (4,1) {}; \node (name_d) at (4.3,1) {$4$};
    \node[vertex] (c) at (2,2) {}; \node (name_c) at (1.7,2) {$2$};
    \draw[->,>=latex] (a) to node [midway,below] {$+0$} (d);
    \draw[->,>=latex] (a) to node [midway,left] {$+0$} (c);
    \draw[->,>=latex] (d) to node [midway,above] {$+1$} (c);
    \node (cap) at (3,-0.7) {(b)};
\end{tikzpicture}
\qquad
\begin{tikzpicture}
    \tikzset{vertex/.style={draw,fill=black,circle,inner sep=2pt,minimum size=3pt}}
    \node[vertex] (a) at (2,0) {}; \node (name_a) at (2,-0.3) {$1$};
    \node[vertex] (b) at (0,1) {}; \node (name_b) at (-0.3,1) {$3$};
    \node[vertex] (c) at (2,2) {}; \node (name_c) at (2,2.3) {$2$};
    \node[vertex] (d) at (4,1) {}; \node (name_d) at (4.3,1) {$4$};
    \draw[->,>=latex] (a) to node  [midway,below] {$+0$} (b);
    \draw[->,>=latex] (b) to [bend right=20] node [midway,below] {$+0$} (c);
    \draw[->,>=latex] (b) to [bend left=20] node [midway,above] {$+1$} (c);
    \draw[->,>=latex] (a) to node [midway,right] {$+0$} (c);
    \draw[->,>=latex] (a) to node [midway,below] {$+0$} (d);
    \draw[->,>=latex] (d) to node [midway,above] {$+1$} (c);  
    \node (cap) at (2,-0.7) {(c)};
\end{tikzpicture}
\caption{(a), (b) $2$-realizable $\mathbb{Z}$-labelled graphs. (c) The union of (a) and (b), which is not $2$-realizable.}
\label{fig:counter}
\end{figure}
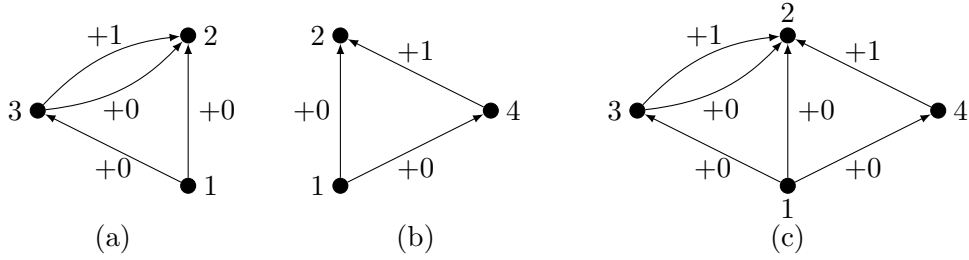
We remark that balancedness of one of the $\mathbb{Z}$-labelled graphs in Lemma~\ref{lem:bk-sum} is necessary, i.e., even if $(G_1,z_1)$ and $(G_2,z_2)$ are $d$-realizable and $(G_1,z_1) \cap (G_2,z_2) =(K_k,0)$ for some $k$, $(G_1,z_1) \cup (G_2,z_2)$ may not be $d$-realizable.
For example, consider $\mathbb{Z}$-labelled graphs (a), (b) in Figure~\ref{fig:counter}. The intersection of (a) and (b) is $(K_2,0)$, and by Proposition~\ref{prop:Kn} (i), (a) and (b) are $2$-realizable. On the other hand, their union shown in Figure~\ref{fig:counter} (c) is not $2$-realizable since if an edge between $1$ and $4$ is contracted, it becomes a graph $(K_{3}^{\bullet\bullet},z)$ for some labelling $z$, which is not $2$-realizable by Proposition~\ref{prop:Kn} (ii).

We provide another operation which also preserves $d$-realizability.
If $(\hat{G},z)=(\hat{G_1},z_1) \cup (\hat{G_2},z_2)$ and $\hat{V}(\hat{G}_1) \cap \hat{V}(\hat{G}_2) = \emptyset$, $(\hat{G},z)$ is called the disjoint union of $(\hat{G_1},z_1)$ and
$(\hat{G_2},z_2)$. For convenience, we consider that an empty $\mathbb{Z}$-labelled graph is a complete graph and its multiplicity graph is spanning.
\begin{lemma} \label{lem:UR-sum}
    Let $(\hat{G},z)=(\hat{G_1},z_1) \cup (\hat{G_2},z_2)$ and let $(\hat{H},w)=(\hat{G_1},z_1)\cap(\hat{G_2},z_2)$.
    Suppose that $\si(\hat{H})$ is a complete graph and the multiplicity graph of $\hat{H}$ is spanning.
    Then we have $\rd(\hat{G},z)=\max\{\rd(\hat{G_1},z_1), \rd(\hat{G_2},z_2)\}$.
\end{lemma}
\begin{proof}
    By the same argument in the first part of the proof of Proposition~\ref{lem:bk-sum}, it suffices to prove $\rd(\hat{G},z) \leq \max\{\rd(\hat{G_1},0), \rd(\hat{G_2},z_2)\}$.
    The case when $(\hat{H},w)$ is empty is trivial.
    Suppose that $(\hat{H},w)$ is not empty. Let $H$ be the lift graph of $(\hat{H},w)$.
    Note that, by Proposition \ref{prop:Kn} (ii), every periodic realization of $H$ is universally rigid.
    Now the proof follows the same line as Lemma~\ref{lem:bk-sum}. 
    We use the same notation as in the proof of Lemma~\ref{lem:bk-sum}.
    For a given periodic framework $(G,p)$, $(G_i,p|_{\hat{V_i}})$ admits an equivalent framework $(G_i,q_i)$ for $i=1,2$ in $\mathbb{R}^d$, where $d=\max\{\rd(\hat{G_1},z_1),\rd(\hat{G_2},z_2)\}$.
    Then by the periodical universal rigidity of $(H,p|_{V(H)})$, one can take an orthogonal matrix $A \in O(\mathbb{R}^d)$ and $t \in \mathbb{R}^d$ satisfying $q_2(v)=Aq_1(v) +t$ for all $v \in \hat{U}$ and $\ell_2=A\ell_1$, where $\ell_i$ is the lattice vector of $(G_i,q_i)$.
    Then, one can check that the same $(G,q)$ given in the proof of Lemma~\ref{lem:bk-sum} is equivalent to $(G,p)$.
\end{proof}

A special case of the situation in Lemma~\ref{lem:UR-sum} is when $\hat{V}_1 \cap \hat{V}_2$ is a singleton. In this case, the operation to take their union is called a \emph{$1$-sum}.

\subsection{Characterization for $d=1$}
Let $\mathcal{F}(K_{2}^\bullet)$ be the set of all simple $\mathbb{Z}$-labelled graphs of the form $(K_{2}^\bullet, z)$ for some $z$.
\begin{theorem} \label{thm:d=1}
    For a $\mathbb{Z}$-labelled graph $(\hat{G},z)$, the followings are equivalent:
    \begin{enumerate}
        \item[(i)] $\rd(\hat{G},z) \leq 1$.
        \item[(ii)] $(\hat{G},z)$ has no minor isomorphic to a $\mathbb{Z}$-labelled graph in $\mathcal{F}(K_{2}^\bullet) \cup \{(K_3,0)\}$.
        \item[(iii)] $(\hat{G},z)$ is obtained by a sequence of disjoint unions and $1$-sums of $\mathbb{Z}$-labelled graphs in $\mathcal{F}$, where $\mathcal{F}$ consists of all $\mathbb{Z}$-labelled graphs isomorphic to either $(K_2,0)$ or a $\mathbb{Z}$-labelled graph with single vertex.
        \item[(iv)] $\hat{G}$ does not contain a cycle other than selfloops.
    \end{enumerate}
\end{theorem}
\begin{proof}
    By Proposition~\ref{prop:Kn}, $\rd(K_2^\bullet,z)=2$ for any labelling $z$ and $\rd(K_3,0)=2$, so by minor monotonocity (i) implies (ii).
    By Proposition~\ref{prop:Kn}, $\rd(K_2,0)=1$ and the realizable dimension of a $\mathbb{Z}$-labelled graph with single vertex is at most $1$. So one can inductively check that (iii) implies (i) using Lemma~\ref{lem:bk-sum} and Lemma~\ref{lem:UR-sum}.
    Since any cycle (except for selfloops) in $\hat{G}$ can be contracted to either $(K_2^\bullet,z) \in \mathcal{F}(K_2^\bullet)$ or $(K_3,0)$, (ii) implies (iv).
    Finally if $(\hat{G},z)$ satisfies (iv), $\si(\hat{G})$ is a forest and the multigraph $\hat{G}$ has no parallel edge. Hence each connected component is obtained by a sequence of $1$-sums of $\mathbb{Z}$-labelled graphs in $\mathcal{F}$.
\end{proof}
We remark that as the minor relation between $\mathbb{Z}$-labelled graphs is equivalent to the minor relation between lift graphs, Theorem~\ref{thm:main} (i) immediately follows from the equivalence of (i) and (ii) in Theorem~\ref{thm:d=1}.
Given a $\mathbb{Z}$-labelled graph $(\hat{G},z)$, the cycle-free condition (iv) is checkable first by checking if there are any parallel edges, and then by checking if $\si(\hat{G})$ contains a cycle. This gives a polynomial time algorithm to check if a given $\mathbb{Z}$-labelled graph is $1$-realizable.

%

\subsection{Characterization for $d=2$}
Recall that $K_3^{\bullet \bullet}$ is a directed multigraph on $\{1,2,3\}$ whose edge set consists of a single edge from $1$ to $2$ and parallel edges between $\{1,3\}$ and $\{2,3\}$.
Let $\mathcal{F}(K_{3}^{\bullet\bullet})$ be the set of all simple $\mathbb{Z}$-labelled graphs of the form $(K_{3}^{\bullet\bullet},z)$ for some edge labelling $z$.
\begin{lemma} \label{lem:d=2}
    If a $\mathbb{Z}$-labelled graph $(\hat{G},z)$ has no minor isomorphic to a $\mathbb{Z}$-labelled graph in $\mathcal{F}(K_{3}^{\bullet\bullet}) \cup \{(K_4,0)\}$, then a simple graph $\si(\hat{G})$ does not have $K_4$ as a minor.
\end{lemma}
\begin{proof}
    Suppose, on the contrary, that $K_4$ is a minor of a simple graph $\si(\hat{G})$. 
    Then $(\hat{G},z)$ has a $\mathbb{Z}$-labelled minor $(K_4,w)$.
    As $(K_4,0)$ is not a minor of $(\hat{G},z)$, $(K_4,w)$ is not balanced.
    We can further say that $(K_4,w)$ has at least two unbalanced cycles since the sum of labels along four cycles $1\rightarrow 2\rightarrow 3\rightarrow 1$, $2\rightarrow 3\rightarrow 4\rightarrow 2$, $3\rightarrow 4\rightarrow 1\rightarrow 3$, $4\rightarrow 1\rightarrow 2\rightarrow 4$ is zero, where $1,2,3,4$ are vertices of $(K_4,w)$.
    Then by contracting the common edge of two unbalanced cycles, one obtains $(K_3^{\bullet\bullet},z)$ for some simple labelling $z$, which is a contradiction.
\end{proof}

We are now ready to characterize $2$-realizable $\mathbb{Z}$-labelled graphs.
\begin{theorem} \label{thm:d=2}
    For a $\mathbb{Z}$-labelled graph $(\hat{G},z)$, the followings are equivalent:
    \begin{enumerate}
        \item[(i)] $\rd(\hat{G},z) \leq 2$.
        \item[(ii)] $(\hat{G},z)$ has no minor isomorphic to a $\mathbb{Z}$-labelled graph in $\mathcal{F}(K_{3}^{\bullet\bullet}) \cup \{(K_4,0)\}$.
        \item[(iii)] $(\hat{G},z)$ is a subgraph of a $\mathbb{Z}$-labelled graph obtained by a sequence of disjoint unions, $1$-sums and balanced $2$-sums of $\mathbb{Z}$-labelled graphs in $\mathcal{F}$, 
        where $\mathcal{F}$ consists of all $\mathbb{Z}$-labelled graphs isomorphic to either $(K_3,0)$ or a $\mathbb{Z}$-labelled graph with at most two vertices.
    \end{enumerate}
\end{theorem}
\begin{proof}
    By Proposition \ref{prop:Kn}, we have $\rd(K_3^{\bullet\bullet},z)=3$ for any $z$ and $\rd(K_4,0)=3$, so by minor monotonicity, (i) implies (ii).
    By Proposition \ref{prop:Kn}, $\rd(K_3,0)=2$ holds, and $\mathbb{Z}$-labelled graphs with at most two vertices are $2$-realizable. So one can inductively check that (iii) implies (i) by Lemma~\ref{lem:bk-sum} and Lemma~\ref{lem:UR-sum}. It remains to prove a combinatorial statement that (ii) implies (iii). For this, we may assume that $\hat{G}$ does not have a selfloop.
    We prove this by induction on $|\hat{V}|$. When $|\hat{V}|=1,2,3$, one can directly check that the statement holds from Proposition \ref{prop:Kn}.
    Suppose, on the contrary, $(\hat{G},z)$ satisfies (ii), but not (iii). 
    Then $\si(\hat{G})$ is $2$-connected, since otherwise $(\hat{G},z)$ is a disjoint union or a $1$-sum of graphs satisfying (ii) with smaller vertices, so by induction hypothesis, $(\hat{G},z)$ satisfies (iii), which is a contradiction.
    By Lemma~\ref{lem:d=2}, $\si(\hat{G})$ does not have $K_4$ as a minor, so $\si(\hat{G})$ has a vertex $v \in \hat{V}$ with degree two in $\si(\hat{G})$ (cf.~\cite{diestel2016graph}).
    Let $x, y$ be two neighbors of $v$ in $\si(\hat{G})$.
    By $2$-connectivity, there is a path in $\si(\hat{G})$ between $x$ and $y$ not passing through $v$. 
    Hence to exclude $(K_{3}^{\bullet\bullet},z)$-minor, either multiplicity between $v$ and $x$ or multiplicity between $v$ and $y$ in $\hat{G}$ must be one.

    Suppose that both multiplicities are one in $\hat{G}$. Let $e_1$ (resp. $e_2$) be an edge between $v$ and $x$ (resp. $v$ and $y$). Switch at $x$ and $y$ so that both labels of $e_1$ and $e_2$ are zero.
    Then $(\hat{G},z)$ is a subgraph of a balanced $2$-sum of $(K_3,0)$ and $(\hat{G}/e_1,z)$. As $(\hat{G}/e_1,z)$ satisfies (ii) and the number of vertices is smaller than $|\hat{V}|$, it satisfies (iii) by induction hypothesis. Then $(\hat{G},z)$ also satisfies (iii), which is a contradiction.
    
    So we may assume that, in $\hat{G}$, the multiplicity between $v$ and $y$ is $1$ and the multiplicity between $v$ and $x$ is at least $2$.
    We claim that $(\hat{G}-v,z)$ is balanced. 
    To see this, suppose, on the contrary, that $(\hat{G}-v,z)$ contains an unbalanced cycle $C$. 
    By Menger's theorem, in $\si(\hat{G})$, there exist vertex disjoint paths $P_1$ and $P_2$ such that $P_1$ connects $u_1 \in V(C)$ and $x$, $P_2$ connects $u_2 \in V(C)$ and $y$, and $V(P_i) \cap V(C)=\{u_i\}$ for $i=1,2$.
    Contract $P_1$ and $P_2$, and contract $C$ to a pair of parallel edges. Then we obtain a $(K_{3}^{\bullet\bullet},z)$-minor, which is a contradiction.
    Hence $(\hat{G}-v,z)$ is balanced. 
    
    Switch all the labels of edges in $\hat{E}[\hat{G}-v]$ to $0$. We prove that $(\hat{G},z)$ satisfies (iii). For this, by adding an edge if necessary, we may assume that $(x,y;0) \in \hat{E}(\hat{G})$.
    $(\hat{G}-v,0)$ is balanced and $\si(\hat{G}-v)$ is an ordinary graph minor of $\si(\hat{G})$, which does not contain $K_4$ as a minor. Hence $(\hat{G}-v,0)$ satisfies (ii).
    Let $(\hat{G}[\{x,y,v\}], z)$ be a $\mathbb{Z}$-labelled graph in which the underlying graph is an induced directed multigraph of $\hat{G}$ on $\{x,y,v\}$ and the edge labelling is restricted to the remaining edges. Then it is a $\mathbb{Z}$-labelled graph minor of $(\hat{G},z)$, so it satisfies (ii).
    By induction hypothesis, both $(\hat{G}-v,0)$ and $(\hat{G}[\{x,y,v\}], z)$ satisfy (iii). Since $(\hat{G},z)$ is a balanced $2$-sum of $(\hat{G}-v,0)$ and $(\hat{G}[\{x,y,v\}],z)$, $(\hat{G},z)$ satisfies (iii). But this contradicts our assumption that $(\hat{G},z)$ does not satisfy (iii). Thus we have confirmed that (ii) implies (iii).
\end{proof}
We remark that Theorem~\ref{thm:main} (ii) immediately follows from Theorem~\ref{thm:d=2}. By following the proof of Theorem~\ref{thm:d=2}, we obtain the following polynomial time algorithm to check $2$-realizability of $\mathbb{Z}$-labelled graphs. Given a $\mathbb{Z}$-labelled graph $(\hat{G},z)$, remove all the selfloops. If the degree of every vertex in $\si(\hat{G})$ is at least $3$, it is not $2$-realizable. Otherwise, if $\si(\hat{G})$ is not connected or it has a size one separator, recursively check the smaller $\mathbb{Z}$-labelled graphs. Otherwise $\si(\hat{G})$ has a degree two vertex $v$ with neighbors $x,y$. If the multiplicity between $v$ and $x$ the multiplicity betwenn $v$ and $y$ are both $2$, it is not $2$-realizable. If both multiplicities are one, contract $vx$ and check recursively. In the remaining case, delete $v$ from $(\hat{G},z)$ and check if $(\hat{G}-v,z)$ is balanced or not, which can be done by taking rooted spanning tree in $\si(\hat{G})$ and by doing switching from the root.

\section{Closing Remark}
    

A characterization of $3$-realizable $\mathbb{Z}$-labelled graphs is open. If there is a size $2$ separator, we can provably reduce the size. So, by~\cite{belk2007realizability}, the characterization of $3$-realizable $\mathbb{Z}$-labelled graphs $(\hat{G},z)$ is reduced to two cases;
the first case is when the treewidth of $\si(\hat{G})$ is at most $3$, and the second case is when $\si(\hat{G})=C_5 \times C_2$ or $V_8$.
As in \cite{so2007semidefinite, so2006semidefinite}, SDP formulation might help to find an equivalent framework in lower affine dimensional space.


Realizable dimension has been defined for other metric spaces such as spherical metric space \cite{laurent2014new,laurent2014positive} and $\ell_p$-space \cite{fiorini2017excluded,sitharam2014flattenability}.
In~\cite{laurent2014new,laurent2014positive}, a spherical variant $\gd^=$ of realizable dimension is introduced, and graphs with $\gd^=$ at most $2,3$, or $4$ are characterized.
They also exploit SDP formulation and deduce a relationship between $\gd^=$ and Colin de Verdi\`{e}re number $\nu^=$.
It is not clear if those results can be extended to $\mathbb{Z}$-labelled graphs.

\paragraph{Acknowledgement} 
This work was supported by JST PRESTO Grant Number JPMJPR2126, JST ERATO Grant Number JPMJER1903, and JST ACT-X Grant Number JPMJAX2009, Japan.

\bibliographystyle{plain}
\bibliography{myreference}

\end{document}